\documentclass[11pt]{article}

\usepackage{amsmath, amssymb, amsthm}
\usepackage{geometry}
\usepackage{authblk}
\usepackage{setspace}
\newtheorem{lemma}{Lemma}
\newtheorem{theorem}{Theorem}
\newtheorem{proposition}{Proposition}
\newtheorem{definition}{Definition}
\usepackage{comment}
\usepackage{color}
\usepackage{amsmath}
\usepackage{cleveref}
\title{\bf Penalisation of Two-Dimensional  Brownian Motion}

\author[1]{Joseph Najnudel}
\author[1]{Thammadol Tansrivorarat}
\affil[1]{School of Mathematics, University of Bristol,  UK}
\date{\today} 

\begin{document}

\maketitle

\begin{abstract}
We study a penalisation problem for two-dimensional Brownian motion. Starting from the Wiener measure, we
consider a family of probability measures obtained by weighting paths by a nonnegative functional $F_t$ depending on $t \geq 0$, $F_t$ being measurable with respect to the $\sigma$-algebra generated by the path up to time $t$.   
Under suitable assumptions on the penalisation process, we establish the weak convergence of these measures when $t \rightarrow \infty$. 
The limiting law is identified explicitly in terms of a  $\sigma-$finite measure $\mathbf{W}^{(2)}$, which admits a path
decomposition involving the last hitting time of a circle. This decomposition plays a central role in the analysis
and yields a martingale representation of the limiting measure.
where ordering and local time techniques are no longer available. The proofs rely on Laplace transform
methods and Tauberian theorems, which replace excursion-theoretic tools and allow a precise identification of
the limiting measure and its structural properties.
\end{abstract}

\medskip
\noindent

\section{Introduction}

 Penalisation problems are a classical topic in probability theory, originating from the study of
how a stochastic process is modified when its law is weighted by a suitable functional of the path.

A large number of articles have been written by Roynette, Vallois and Yor on penalisation of one-dimensional Brownian motion, included many examples in \cite{RVY06a, RVY06b, RY08, RVY09}, summarized in \cite{RVY06}
and \cite{RY09Book}.
Brownian penalisation has found numerous applications. In particular, it allows to:
\begin{enumerate}
    \item analyze rare events, such as Brownian paths that spend an atypically long time in a given region or avoid certain sets;
    \item investigate the long-time behavior of Brownian motion under pathwise constraints;
    \item construct new probability measures that encode conditioned or constrained versions of Brownian motion.
\end{enumerate}
Penalisation of multidimensional Brownian motion has also been explored by Roynette, Vallois and Yor in \cite{RVY09ESAIM}. Extending the one-dimensional theory to higher dimensions is nontrivial, as the absence of an ordering and of
one-dimensional local times prevents the direct use of the techniques that play a central role in the RVY
framework. In particular, new methods are required to identify the limiting measures and their structural
properties.

In Najnudel, Roynette and Yor \cite{NRY09}, a large part of the one-dimensional penalisation examples studied before have been unified into a common setting, the limiting distribution being absolutely continuous with respect to a $\sigma$-finite measure $\mathbf{W}$ independent of the weight chosen for penalisation. 
The measure $\mathbf{W}$, which has infinite total mass, is of independent interest, and can be very informally seen as the law of the Brownian motion conditioned to be transient. In \cite{NRY09}, the construction of $\mathbf{W}$ has been extended to the settings of the two-dimensional Brownian motion, the linear one-dimensional diffusions, and the recurrent Markov chains on a countable space. However, the general penalisation result in \cite{NRY09} has so far been proven only in the one-dimensional Brownian setting. The purpose of the present article is to extend it to the two-dimensional setting.

Let $W^{(2)}$ denote the standard Wiener measure defined on the filtered probability space
$\Omega = 
\big(\mathcal{C}(\mathbb{R}_{+}, \mathbb{C}),\,
(X_t)_{t \ge 0},\, (\mathcal{F}_t)_{t \ge 0}, \mathcal{F}_{\infty}\big)$,
where $(X_t)_{t \ge 0}$ is the coordinate process, $(\mathcal{F}_t)_{t \ge 0}$ its natural filtration and $\mathcal{F}_{\infty}$ the $
\sigma$-algebra generated by $X_s$ for all $s \geq 0$. Throughout this paper, expectations with respect to $W^{(2)}$ will be denoted by $W^{(2)}(\cdot)$.\\
Let $(F_t)_{t \ge 0}$ be a nonnegative, $(\mathcal{F}_t)_{t \geq 0}$-adapted process such that
\[
0 < W^{(2)}[F_t] < \infty \quad \text{for all } t \ge 0.
\]
Associated with $(F_t)_{t \geq 0}$, we define a family of probability measures $(\mathbb{Q}_t)_{t \ge 0}$ on
$\Omega$ by
\[
\mathbb{Q}_t(\Gamma)
=
\frac{W^{(2)}\!\left(\mathbf{1}_{
\Gamma} F_t\right)}{W^{(2)}(F_t)},
\qquad \Gamma \in \mathcal{F}_{\infty}.
\]
The central question of the penalisation problem is to understand the asymptotic behavior of
$\mathbb{Q}_t$ as $t \to \infty$, and, when possible, to identify a limiting probability measure
$\mathbb{Q}$ describing the long-time behavior of Brownian motion under this weighted law. 

The main results of the present article extend the one-dimensional penalisation results proven in \cite{NRY09} to the two-dimensional setting, for which 
only the construction of the $\sigma$-finite measure similar to $\mathbf{W}$ has been done.

Our approach relies on a path decomposition at the last exit time
from the unit circle under a $\sigma-$finite universal measure, together with a Laplace transform and Tauberian
analysis. Within this framework, we identify a natural class of nonincreasing penalisation functionals that stabilise
after exiting a bounded region, and we prove the convergence of the associated penalised measures. The
limiting laws are described explicitly in terms of the universal $
\sigma$-finite measure constructed in 
\cite{NRY09}. 

The present article is organized as follows: in Section \ref{setting}, we study some properties of the $\sigma$-finite measure given in \cite{NRY09} and we state our main theorem, which is proven in Section \ref{proof}. In Section \ref{discussion},  we discuss limitations of the present approach and outline possible directions for further research.
\section{Properties of the measure $\mathbf{W}^{(2)}$ and statement of the main theorem}
\label{setting}
The main result of this article involves 
a $\sigma$-finite measure $\mathbf{W}^{(2)}$
on $\Omega$, which can be constructed via the following result, immediately deduced from Theorems 2.1.1. and 2.1.2 of \cite{RVY09}. 
\begin{theorem} \label{definitionW2}
Let $q$ be a nonnegative function from $
\mathbb{C}$ to $\mathbb{R}$, compactly supported and such that its integral on $
\mathbb{C}$ is positive and finite. 
Then, there exists a unique continuous, strictly positive function $\varphi_q$ from $\mathbb{C}$ to $\mathbb{R}$ satisfying the Sturm-Liouville equation $\Delta \varphi_q = q 
\varphi_q$ in the sense of Schwartz distribution, and such that the radial derivative of $\varphi_q$ at $z$ is well-defined for $|z|$ large enough 
and equivalent to $1/(\pi |z|)$ when $|z| \rightarrow \infty$. 
Under the Wiener measure $W^{(2)}$, and for 
$$A_s^{(q)} := \int_0^s q(X_u) du$$
the process 
$$\left( M_s^{(2,q)} = \frac{\varphi_q(X_s)}{
\varphi_q(0)} \exp \left(- \frac{1}{2} A_s^{(q)} \right) 
\right)_{s \geq 0}$$
is a martingale with respect to the filtration $(\mathcal{F}_s)_{s 
\geq 0}$, and there exists a probability measure $W_{\infty}^{(2,q)}$ on 
$\Omega$, such that for all $s \geq 0$, its restriction to the $\sigma$-algebra $\mathcal{F}_s$ has density 
$ M_s^{(2,q)}$ with respect to $W^{(2)}$. 
Moreover, under $ W_{\infty}^{(2,q)}$, $A_{\infty}^{(q)}$ is almost surely finite, and the $\sigma$-finite measure $
\mathbf{W}^{(2)}$ given by 
$$\mathbf{W}^{(2)} = \varphi_q(0) \exp 
\left( \frac{1}{2} A_{\infty}^{(q)} 
\right) 
\cdot W_{\infty}^{(2,q)} $$ 
does not depend on $q$.

\end{theorem}

The measure $\mathbf{W}^{(2)}$ is $\sigma$-finite but has an infinite total mass. It satisfies the Brownian scaling property: 
\begin{proposition}
For $a > 0$, let $S_a$ be the scaling operator on paths, given by 
$$(S_a(\omega))_t = a \omega_{t/a^2}$$
for all $\omega \in \mathcal{C}(\mathbb{R}_+, \mathbb{C})$. 
Then, the $\sigma$-finite measure $\mathbf{W}^{(2)}$ is invariant by $S_a$. 

\end{proposition}
\begin{proof}
Let $q$ be a nonnegative function from $
\mathbb{C}$ to $\mathbb{R}$, compactly supported and such that its integral on $
\mathbb{C}$ is positive and finite. 
We define $q_a$ by 
$$q_a(z) := a^2 q (az).$$
The functions $q_a$ satisfy the same general assumptions as $q$. 
The function $\varphi$ from $\mathbb{C}$ to $\mathbb{R}$ given by 
$$\varphi(z) = \varphi_q(az)$$
satisfies, in the sense of Schwartz distributions, 
$$\Delta \varphi (z) = a^2 (\Delta \varphi_q) (az)
= a^2 q(az) \varphi_q(az) = q_a (z) \varphi(z). $$
Moreover, for $|z|$ large enough, the radial derivative of $\varphi$ at $z$ is $a$ times the radial derivative of $\varphi_q$ at $az$, and then it is equivalent to $1/(\pi |z|)$ when $|z| \rightarrow \infty$. 
By uniqueness properties of $\varphi_{q_a}$, we deduce that 
$$\varphi_{q_a} (z) = \varphi(z) =  \varphi_q(az) $$
for all $z \in \mathbb{C}$. 

Hence, 
$$M_{s}^{(2, q_a)} = \frac{\varphi_q(a X_s)}{\varphi_q(0)} 
\exp \left( - \frac{a^2}{2} \int_0^s q(a X_u) du \right)
=  \frac{\varphi_q((S_a(X))_{a^2 s})}{\varphi_q(0)} 
\exp \left( - \frac{1}{2} \int_0^s q((S_a (X))_{a^2 u}) d(a^2 u) \right)
$$
and then 
$$M_{s}^{(2, q_a)} = M_{a^2 s}^{(2,q)} \circ S_a.$$
We deduce, for 
$s \geq 0$ and for a non-negative $\mathcal{F}_s$-measurable random variable $G$, 
$$W^{(2, q_a)} (G) 
= W^{(2)} ( G M_{s}^{(2, q_a)})
=W^{(2)} ( G M_{a^2 s}^{(2, q)} \circ S_a )
= W^{(2)} ( G \circ S_a^{-1})  M_{a^2 s}^{(2, q)} ),$$
where the last equality is due to the Brownian scaling. 
Since $G$ is $\mathcal{F}_s$-measurable, the random variable $G \circ S_a^{-1}$ is measurable with respect to the $\sigma$-algebra generated by 
$X_u \circ S_a^{-1} = a^{-1} X_{u a^2} $ for $u \leq s$, which is 
$\mathcal{F}_{a^2 s}$. Hence, 
$$W^{(2, q_a)} (G) = W^{(2, q)} (G \circ S_a^{-1})$$
for all $s \geq 0$, and for all $\mathcal{F}_s$-measurable, nonnegative random variables $G$, and then for all $\mathcal{F}_{\infty}$-measurable, nonnegative random variables $G$. We deduce
\begin{align*}
\mathbf{W}^{(2)} (G) & = W^{(2, q_a)} \left(G \varphi_{q_a} (0) 
\exp \left(\frac{1}{2} \int_0^{\infty} q_a(X_u) du \right) \right)
\\ & =  W^{(2, q_a)} \left((G \circ S_a^{-1})  \varphi_{q_a} (0) 
\exp \left(\frac{1}{2} \int_0^{\infty} q_a(X_u \circ S_a^{-1}) du \right) \right)
\\ & = W^{(2, q_a)} \left((G \circ S_a^{-1})  \varphi_{q} (0) 
\exp \left(\frac{1}{2} \int_0^{\infty} a^2 q( a (a^{-1} X_{ua^2})) du \right) \right)
 \\ & = W^{(2, q_a)} \left((G \circ S_a^{-1})  \varphi_{q} (0) 
\exp \left(\frac{1}{2} \int_0^{\infty}  q(   X_{s}) ds \right) \right)
\\ & = \mathbf{W}^{(2)} (G \circ S_a^{-1}). 
\end{align*} 

\end{proof}

Another property of $\mathbf{W}^{(2)}$ we use in this article is 
a decomposition involving inverse local times on circles, and a transient stochastic process on the complex plane.  
For $a > 0$, we denote by $C_a$ the circle of center $0$ and radius $a$, and for $t \geq 0$, we define the local time 
$$L_t^{(C_a)} := \underset{\varepsilon \rightarrow 0}{\lim} 
\frac{1}{2 \pi \varepsilon} \int_0^t \mathbf{1}_{a (1 - \varepsilon) \leq |X_s| \leq a(1 + \varepsilon)} ds$$ 
when this limit exists, which occurs almost surely under $W^{(2)}$. 
The inverse local time on $C_a$, which is almost surely right-continuous, is given by 
$$\tau_{\ell}^{(C_a)} = \inf \{t \geq 0, L_t^{(C_a)} > \ell \}.$$
The transient process involved in the decomposition of $\mathbf{W}^{(2)}$
is defined as follows. 

We also denote by \( P_1^{(2, \log)} \) the law of the process \( (R_t)_{ t \geq 0} \) which solves the stochastic differential equation:

\[
R_t = 1 + \beta_t + \int_0^t \frac{ds}{R_s} \left( \frac{1}{2} + \frac{1}{\log R_s} \right)
\]
where $(\beta_t)_{t \geq 0}$ is a one-dimensional Brownian motion starting from 0. 
Here is another description of the process $(R_t)_{t \geq 0}$: we have 

\[
(\log R_t, t \geq 0) \overset{\text{(law)}}{=} (\rho_{H_t}, t \geq 0)
\]
where $(\rho_u)_{u \geq 0}$ is a 3-dimensional Bessel process starting from 0, and 
\begin{equation}H_t := \int_0^t \frac{ds}{R_s^2}. \label{Ht}
\end{equation}
We define $\tilde{P}_1^{(2,\log)}$ as the law of the process 
$(R_t e^{i \alpha_{H_t}})_{t \geq 0}$, where $(R_t)_{t \geq 0}$ follows the probability distribution \( P_1^{(2, \log)} \), 
$(\alpha_t)_{t \geq 0} $ is a Brownian motion starting from $0$ and independent of $(R_t)_{t \geq 0} $ and $H_t$ is given by \eqref{Ht}. For $a > 0$, we define $\tilde{P}_a^{(2,\log)}$ as the image of $\tilde{P}_1^{(2,\log)}$ by $S_a$, i.e. the law of 
$(a X_{t/a^2})_{t \geq 0}$ where $X$ follows the distribution 
$\tilde{P}_1^{(2,\log)}$. 

Finally, for $a, \ell > 0$, we define 
the probability distribution 
$$W^{(2,\tau_{\ell}^{(C_a)})} \circ \tilde{P}_a^{(2,\log)}$$ as the law of the 2-dimensional process $(Y_t)_{t \geq 0}$ satisfying the following properties:

\begin{enumerate}
\item $(Y_t)_{ t \leq \tau_{\ell}^{(C_a)})}$ is a two-dimensional Brownian motion starting from 0 and stopped at \(\tau_{\ell}^{(C_a)}\). 
\item For an independent process $(Z_t)_{t \geq 0}$ following the distribution $\tilde{P}_a^{(2,\log)}$, we have 
$Y_{t +\tau_{\ell}^{(C_a)} } = Z_t e^{i \gamma}$ where $\gamma$ is the argument of $Y_{\tau_{\ell}^{(C_a)}}$. 
\end{enumerate}

The measure $\mathbf{W}^{(2)}$ satisfies the following properties: 
\begin{proposition} \label{decompositionWWbf}
For any $a > 0$, we have 
$$\mathbf{W}^{(2)} = \frac{1}{a^2} \int_0^{\infty} d \ell (W^{(2, \tau_{\ell}^{(C_a)})} \circ \tilde{P}_a^{(2, \log)} ).$$
 Moreover, for all $t \geq 0$ and all bounded, $\mathcal{F}_t$-measurable random variables 
 $G_t$, we have 
$$\mathbf{W}^{(2)} (G_t \mathbf{1}_{g_{C_a} \leq t}) 
= \frac{1}{ \pi} W^{(2)} (G_t \log^+(|X_t|/a))$$
where $g_{C_a}$ is the last hitting time of $C_a$ and $\log^+$ denotes the positive part of the logarithm. 
\end{proposition}
\begin{proof}
For $a = 1$, this result is given in Theorem 2.2.2. of \cite{NRY09}. 
From the definition of the local time, applying $S_a$ gives, for $a > 0$, 
$t \geq 0$, 
$$L_t^{(C_a)} \circ S_a  = a^2 L^{(C_1)}_{t/a^2}$$
when the local times are well-defined, and 
$$\tau_{\ell}^{(C_a)} \circ S_a = a^2 \tau_{\ell/a^2}^{(C_1)}.$$
If we apply $S_a$ to a path $\omega$ stopped at $\tau_{\ell}^{(C_1)}(\omega)$, 
we obtain a new path $S_a(\omega)$ stopped at $a^2 \tau_{\ell}^{(C_1)} (\omega)$, which is equal to $\tau_{\ell a^2}^{(C_a)} (S_a(\omega))$. 
Since $W^{(2)}$ is invariant by $S_a$, the image of $W_0^{(2, \tau_{\ell}^{(C_1)})}$ by $S_a$
is $W_0^{(2, \tau_{\ell a^2}^{(C_a)})}$. 
On the other hand, by definition, the image of $\widetilde{P}_1^{(2, \log)}$ by 
$S_a$ is $\widetilde{P}_a^{(2, \log)}$. 
From the case $a = 1$, we deduce that the image of $\mathbf{W}^{(2)}$ by 
$S_a$ is 
$$\int_0^{\infty} d \ell (W_0^{(2, \tau_{\ell a^2}^{(C_a)})} \circ \widetilde{P}_a^{(2, \log)} ) = 
\frac{1}{a^2}\int_0^{\infty} d \ell (W_0^{(2, \tau_{\ell}^{(C_a)})} \circ \widetilde{P}_a^{(2, \log)} ).
$$
Since we have proven the invariance of $\mathbf{W}^{(2)}$ by $S_a$, we get the first part of the proposition. The invariance of $\mathbf{W}^{(2)}$ also implies, for a bounded, $\mathcal{F}_t$-measurable variable $G_t$, 
$$\mathbf{W}^{(2)} (G_t \mathbf{1}_{g_{C_a} \leq t}) 
= \mathbf{W}^{(2)} ((G_t \circ S_a)( \mathbf{1}_{g_{C_a} \leq t} \circ S_a)) = \mathbf{W}^{(2)} ((G_t \circ S_a)  \mathbf{1}_{g_{C_1} \leq t/a^2}). $$
Now, $G_t \circ S_a$ is in $\mathcal{F}_{t/a^2}$, and then the second part of the proposition, already proven for $a = 1$, gives 
$$\mathbf{W}^{(2)} (G_t \mathbf{1}_{g_{C_a} \leq t}) 
= \frac{1}{ \pi} W^{(2)} ((G_t \circ S_a) \log^+(|X_{t/a^2}|))
= \frac{1}{ \pi} W^{(2)} ((G_t \circ S_a) \log^+|(X_{t} \circ S_a)/a|)$$
Applying the Brownian scaling completes the proof of the proposition.

\end{proof}
Another result we use here is provided by Theorem 2.4.1. of \cite{NRY09}, generalizing Theorem 1.2.1. of \cite{NRY09} from dimension one to dimension two.  
After some change of notation, we get the following:  
\begin{proposition} \label{MtF}
Let $F$ be an integrable random variable defined on the probability space $(\Omega, 
\mathcal{F}_{\infty}, \mathbf{W}^{(2)})$. There exists a continuous martingale $(M_t(F))_{t \geq 0}$, nonnegative if $F \geq 0$, such that for every $t \geq 0$, and every bounded, $\mathcal{F}_t$-measurable random variable $Y_t$, 
$$\mathbf{W}^{(2)} (F Y_t) = W^{(2)} (M_t^{(2)} (F) Y_t).$$
Moreover, $W^{(2)}$-almost surely, 
$$M_t^{(2)} (F) = \mathbf{W}^{(2)} (G)$$
where for all $\omega \in \Omega$ such that $\omega_0 = 0$, 
$$G ((\omega_u)_{u \geq 0}) 
= F((\omega'_u)_{u \geq 0})$$
where $\omega'_u = X_u$ for $0 \leq u \leq t$, and 
$\omega'_u = X_t + \omega_{u-t}$ for $u \geq t$. 

\end{proposition}

We have now the main basic ingredients needed in order to prove the main theorem of the article. 
The class of functionals for which the theorem applies are given by the following definition, which is similar to the definition given in \cite{NRY09} in the one-dimensional setting. 

\begin{definition} Let $(F_{t})_{t \geq 0}$ denote an $(\mathcal{F}_t)$-adapted, nonnegative process. We shall say that this process belongs to the class $\mathcal{C}$ if and only if the following conditions hold:

\begin{enumerate}
    \item  $(F_{t})_{t \geq 0}$ is a nonnegative and nonincreasing process. In particular, since $0 \leq F_{t} \leq F_{0}$ and $F_{0}$ is a.s. constant ($\mathcal{F}_0$ is trivial), this process is $W^{(2)}$-almost surely bounded by a constant $C = F_{0}$.
     \item 
    There exists $a>0$ such that for every $t \geq \sigma_a$, where $
    \sigma_a := \sup \{ t \geq 0 :  |X_{t}|\le a \}$, 
    we have $$F_{t} = F_{\sigma_a} = F_{\infty} := \underset{s \rightarrow \infty}{\lim} F_s.$$
    \item We have 
    $$ \mathbf{W}^{(2)}(F_{\infty}) = \mathbf{W}^{(2)}(F_{\sigma_a}) < \infty.$$ 
    \end{enumerate}
\end{definition}
 For example, bounded functions $\varphi$ of the local times at time $t$ on finitely many circles, and of functionals $A_t^{(q)}$ as defined in Theorem \ref{definitionW2}, for finitely many choices of functions $q$,  belong to the class $\mathcal{C}$ if $\varphi$ is bounded and nonincreasing in each of the arguments. Notice that for this fact to be fully rigorous, we can consider a version of the local time which is adapted,  defined everywhere as a value in $[0, \infty]$ and nondecreasing, for example 
$$L_t^{(C_a)} := \underset{\varepsilon \rightarrow 0}{\lim \inf} 
\frac{1}{2 \pi \varepsilon} \int_0^t \mathbf{1}_{a (1 - \varepsilon) \leq |X_s| \leq a(1 + \varepsilon)} ds$$ 
for the local time on the circle of center $0$ and radius $a$. 
With this definition, the main result of the article is the following: 
\begin{theorem}\label{TheMainResult}
  Let $(F_{t})_{t \geq 0}$ be a process belonging to $\mathcal{C}$, 
 such that $\mathbf{W}^{(2)} (F_{\infty}) > 0$. Then, 
 for all $t$, $0 < W^{(2)} (F_t) < \infty$, and
 we can define the sequence of probability measures $(\mathbb{Q}_t)_{t \geq 0}$ on $\Omega$ by 
  $$\mathbb{Q}_t(\Gamma) = \frac{W^{(2)} (\mathbf{1}_{\Gamma} F_t)}{ W^{(2)} (F_t)}$$
  for all $\Gamma \in \mathcal{F}_{\infty}$. 
Moreover, the following penalisation result holds: for all $s \geq 0$, $\Gamma_s \in \mathcal{F}_s$, 
$$\mathbb{Q}_t (\Gamma_s) \underset{t \rightarrow \infty}{\longrightarrow} 
W^{F}_{\infty} (\Gamma_s)
$$
where $W^{F}_{\infty}$ is the probability measure on $\Omega$ given by 
$$W^{F}_{\infty} := \frac{F_{\infty}}{\mathbf{W}^{(2)}(F_{\infty})} \cdot \mathbf{W}^{(2)}.$$
Moreover, restriction of $W^{F}_{\infty}$ to $\mathcal{F}_s$ has density $M^{(2)}_t(F_{\infty})/\mathbf{W}^{(2)} (F_{\infty})$ with respect to $W^{(2)}$, where $M^{(2)}_t(F_{\infty})$ is defined in Proposition \ref{MtF}. 
\end{theorem}
The next section is devoted to the proof of this theorem.

\section{Proof of the main theorem}
\label{proof}
The main ingredient of the proof consists in computing the asymptotic of the denominator in the definition of $\mathbb{Q}_t$. We get the following: 
\begin{proposition} \label{limitlogt}
Let $(F_{t})_{t \geq 0}$ be a process belonging to $\mathcal{C}$. Then
   
  \begin{equation}\frac{\log t}{2 \pi} W^{(2)} (F_t) \underset{t \rightarrow \infty}{\longrightarrow} \mathbf{W}^{(2)}(F_{\infty})
  \label{asymptoticslog}
  \end{equation}
\end{proposition}
\begin{proof}
If the result is proven for $\mathbf{W}^{(2)}(F_{\infty}) > 0$, then we can extend it to 
the case $\mathbf{W}^{(2)}(F_{\infty}) = 0$ by 
subtracting the limits corresponding to the functionals 
$(F_t + G_t)_{t \geq 0}$ and $(G_t)_{t \geq 0}$, where 
$(G_t)_{t \geq 0}$ is any functional in the class $\mathcal{C}$ such that  
$\mathbf{W}^{(2)}(G_{\infty}) > 0$, for example 
$G_t = e^{-A_{t}^{(q)}/2}$, with the assumptions and notation of 
Theorem \ref{definitionW2}. 
From now, we can then assume $\mathbf{W}^{(2)}(F_{\infty}) > 0$.

  We write:
$$F_t = F_{t}^{(1)} + F_{t}^{(2)} +F_{t}^{(3)},$$
where 
$$F_t^{(1)}  = F_t \frac{\log^{+}|X_{t}/a|}{1+\log^{+}|X_{t}/a|},$$
$$F_t^{(2)}  = F_t \frac{\log^{+}|X_{t}/a|}{(1+\log^{+}|X_{t}/a|)^2},$$
$$F_t^{(3)}  = F_t \frac{1}{(1+\log^{+}|X_{t}/a|)^2}.$$
For $j \in \{1,2\}$, $\lambda > 0$, we have, 
from Proposition \ref{decompositionWWbf} and the fact that $(F_t)_{t \geq 0}$ is in class $\mathcal{C}$,
which implies $F_t = F_{g_{C_a}} = F_{\sigma_a} = F_{\infty}$
for $g_{C_a} \leq t$, 
\begin{align*}
\int_0^\infty e^{-\lambda t}{W}^{(2)}(F_t^{(j)}) dt
&= 
\int_0^\infty e^{-\lambda t} \, W^{(2)}\left(\frac{\log^+ |X_t/a|}{(1 + \log^+ |X_t/a|)^j} F_t \right)\\
&=\pi \int_0^\infty e^{-\lambda t}
\mathbf{W}^{(2)}\left[
F_t \frac{1_{g_{C_a} \leq t}}{(1 + \log^+ |X_t/a|)^j}
\right]\\
&=\pi\int_0^\infty e^{-\lambda t}
\mathbf{W}^{(2)}\left[
F_{g_{C_a}} \frac{1_{g_{C_a} \leq t}}{(1 + \log^+ |X_t/a|)^j} 
\right]\\
&=\pi \mathbf{W}^{(2)}\left[ F_{g_{C_a}} e^{-\lambda g_{C_a}} \int_0^\infty \frac{e^{-\lambda u}}{(1 + \log^+|X_{g_{C_a}+u}/a|)^j} du \right] \\
&=\frac{\pi}{a^2} 
\int_{0}^{\infty} 
d \ell 
W^{(2)} ( F_{\tau_{\ell}^{(C_a)}}  e^{-\lambda \tau_{\ell}^{(C_a)}} )
 \tilde{P}_a^{(2, \log)} \left[ 
 \int_0^{\infty} \frac{e^{-\lambda u} du}{(1 +
 \log|X_{u}/a|)^j}
 \right]. 
\end{align*}
 The last line is due to the fact that 
 under 
 $W^{(2, \tau_{\ell}^{(C_a)})} \circ \tilde{P}_a^{(2, \log)}$, $\tau_{\ell}^{(C_a)}$ is the last hitting time of $C_a$, whereas
 $(|X_{g_{C_a} +u} |)_{u \geq 0}$ is independent of 
 the trajectory of $X$ up to time $\tau_{\ell}^{(C_a)}$, and follows the same law as $(|X_u|)_{u \geq 0}$ under $\tilde{P}_a^{(2, \log)}$. 
 Proposition \ref{decompositionWWbf} also implies
 
 $$ \mathbf{W}^{(2)}\left[ F_{\infty} e^{-\lambda g_{C_a}} \right] = \mathbf{W}^{(2)}\left[ F_{g_{C_a}} e^{-\lambda g_{C_a}} \right]
 = \frac{1}{a^2} 
\int_{0}^{\infty} 
d \ell 
W^{(2)} ( F_{\tau_{\ell}^{(C_a)}}  e^{-\lambda \tau_{\ell}^{(C_a)}} ).$$
Therefore,
\begin{equation}\int_0^\infty e^{-\lambda t}{W}^{(2)}(F_t^{(j)}) dt 
= \pi 
\mathbf{W}^{(2)}\left[ F_{\infty} e^{-\lambda g_{C_a}} \right] \tilde{P}_a^{(2, \log)} \left[ 
 \int_0^{\infty} \frac{e^{-\lambda u } du}{(1 +
 \log|X_{u}/a|)^j} \right].
 \label{tauberianintegral}
 \end{equation}
The computation above is also true when $F_t$ is replaced by $1$. We get 
\begin{equation}\int_0^\infty e^{-\lambda t}{W}^{(2)}
\left( \frac{\log^+ |X_t/a|}{(1 + \log^+|X_t/a|)^j}
\right) 
dt = \pi \mathbf{W}^{(2)}\left[ e^{-\lambda g_{C_a}} \right] \tilde{P}_a^{(2, \log)} \left[ 
 \int_0^{\infty} \frac{e^{-\lambda u } du}{(1 +
 \log|X_{u}/a|)^j} \right]. \label{Ft=1}
 \end{equation}
Under $W^{(2)}$, $|X_t|^2$ is $t$ times a chi-square random variable $\chi_2$ with two degrees of freedom, which implies that 
$$\log |X_t/a| = \frac{1}{2} \log t - \log a + \frac{1}{2} \log \chi_2$$
in distribution. 
We deduce 
\begin{align*}\int_0^\infty e^{-\lambda t}{W}^{(2)}
\left( \frac{\log^+ |X_t/a|}{(1 + \log^+|X_t/a|)^j}
\right) 
dt 
& = \frac{1}{\lambda} 
\mathbb{E} \left[
\frac{(( \log \mathbf{e}_{\lambda})/2 - \log a + (\log \chi_2)/2)^+} {(1 + ( (\log \mathbf{e}_{\lambda})/2 - \log a + (\log \chi_2)/2)^+)^j } \right] 
\\ & = \frac{1}{\lambda} 
\mathbb{E} \left[
\frac{(( \log \mathbf{e}_{1} - \log \lambda)/2 - \log a + (\log \chi_2)/2)^+} {(1 + ( (\log \mathbf{e}_{1} - \log \lambda)/2 - \log a + (\log \chi_2)/2)^+)^j } \right] 
\end{align*}
where $\mathbf{e}_{\lambda}$ is an exponential random variable with parameter $\lambda$, independent of $\chi_2$. 
Hence, for $\lambda \in (0,1)$, 
$$ \lambda \left(- \frac{1}{2} \log \lambda \right)^{j-1}
\int_0^\infty e^{-\lambda t}{W}^{(2)}
\left( \frac{\log^+ |X_t/a|}{(1 + \log^+|X_t/a|)^j}
\right) 
dt = \mathbb{E} \left[
\frac{( 1 + V/(-\log \lambda) )^+} {(2/(- \log \lambda) +  (1 + V/(-\log \lambda))^+)^j } \right] 
$$
for $$V := \log \mathbf{e}_1 - 2 \log a +  \log \chi_2.$$
Now, 
$$\mathbb{E} \left[ \mathbf{1}_{V \geq (\log \lambda)/2 }
\frac{( 1 + V/(-\log \lambda) )^+} {(2/(- \log \lambda) +  (1 + V/(-\log \lambda))^+)^j } \right]
\underset{\lambda \rightarrow 0}{\longrightarrow} 
1$$
by dominated convergence, since the quantity inside the expectation is bounded by $1$ for $j = 1$ and by $2$ for $j = 2$. On the other hand, 
$$\mathbb{E} \left[ \mathbf{1}_{V <  (\log \lambda)/2 }
\frac{( 1 + V/(-\log \lambda) )^+} {(2/(- \log \lambda) +  (1 + V/(-\log \lambda))^+)^j } \right]
\leq \mathbb{E} \left[ \mathbf{1}_{V <  (\log \lambda)/2 }
\frac{1/2} {(2/(- \log \lambda))^j } \right]
$$
 tends to zero, since 
 $$\mathbb{P} ( V < (\log \lambda)/2)
 = \mathbb{P} ( \mathbf{e}_1 \chi_2 < a^2 \sqrt{\lambda})
 \leq \mathbb{P} (\mathbf{e}_1 < a \lambda^{1/4})
 + \mathbb{P} (\mathbf{e}_{1/2} < a \lambda^{1/4})
 = \mathcal{O}(a \lambda^{1/4})
 $$
is negligible with respect to $(- \log \lambda)^j$ when $\lambda \rightarrow 0$. 
We deduce 
$$\lambda \left(- \frac{1}{2} \log \lambda \right)^{j-1}
\int_0^\infty e^{-\lambda t}{W}^{(2)}
\left( \frac{\log^+ |X_t/a|}{(1 + \log^+|X_t/a|)^j}
\right) 
dt \underset{\lambda \rightarrow 0}{\longrightarrow} 
1. $$
 Combining with \eqref{Ft=1}, we get
$$ \pi \mathbf{W}^{(2)}\left[ e^{-\lambda g_{C_a}} \right] \tilde{P}_a^{(2, \log)} \left[ 
 \int_0^{\infty} \frac{e^{-\lambda u } du}{(1 +
 \log|X_{u}/a|)^j} \right]
 \sim \frac{1}{\lambda ( - (\log \lambda)/2)^{j-1} }
 $$
 when $\lambda \rightarrow 0$. 
 Then, \eqref{tauberianintegral} implies 
$$
\int_0^\infty e^{-\lambda t}{W}^{(2)}(F_t^{(j)}) dt 
\sim_{\lambda \rightarrow 0}  \frac{1}{\lambda ( - (\log \lambda)/2)^{j-1} } \frac{\mathbf{W}^{(2)}\left[ F_{\infty} e^{-\lambda g_{C_a}} \right]}{\mathbf{W}^{(2)}\left[ e^{-\lambda g_{C_a}} \right]}.$$
On the other hand, 
\begin{align*}
    \mathbf{W}^{(2)}[e^{-\lambda g_{c_a}}]
&=
\lambda\mathbf{W}^{(2)}\left[\int_0^\infty 1_{g_{C_a} \leq u} e^{-\lambda u} \, du \right]\\
&=\lambda \int_0^\infty e^{-\lambda u} \mathbf{W}^{(2)}\left(g_{C_a} \le u\right) du\\
&=\frac{\lambda}{\pi} \int_0^\infty e^{-\lambda u} W^{(2)} \left(\log^{+}|X_u/a|\right) du\\
& = \frac{1}{\pi} \mathbb{E} [ ((\log \mathbf{e}_{\lambda})/2 - \log a + (\log \chi_2)/2)^+ ]
\\ & = \frac{1}{\pi} \mathbb{E} [ (V - \log \lambda)^+/2]
= \frac{-\log \lambda}{2 \pi} 
+ \mathcal{O} (\mathbb{E} [|V|])
\end{align*}
for $\lambda \in (0,1)$. Since $V$ is integrable, we deduce 
$$\int_0^\infty e^{-\lambda t}{W}^{(2)}(F_t^{(j)}) dt 
\sim_{\lambda \rightarrow 0}  \frac{\pi}{\lambda ( - (\log \lambda)/2)^{j} } \mathbf{W}^{(2)}\left[ F_{\infty} e^{-\lambda g_{C_a}} \right]$$
By monotone convergence, 
$$\mathbf{W}^{(2)}\left[ F_{\infty} e^{-\lambda g_{C_a}} \right] 
\underset{\lambda \rightarrow 0}{\longrightarrow}
\mathbf{W}^{(2)}\left[ F_{\infty}  \right].$$
Since $(F_t)_{t \geq 0}$ is in the class $\mathcal{C}$, the limit is non-zero and finite. 
Hence,  
$$\int_0^\infty e^{-\lambda t}{W}^{(2)}(F_t^{(j)}) dt 
\sim_{\lambda \rightarrow 0}  \frac{\pi}{\lambda ( - (\log \lambda)/2)^{j} } \mathbf{W}^{(2)}\left[ F_{\infty} \right]$$
for $j \in \{1,2\}$. 
For $j = 3$, we crudely bound $F_t$ by $F_0$, which is almost surely equal to a constant $C \geq 0$. 
We get 
\begin{align*}\int_0^\infty e^{-\lambda t}{W}^{(2)}(F_t^{(3)}) dt & 
\leq C \int_0^\infty e^{-\lambda t} {W}^{(2)}
\left( \frac{1}{ (1 + \log^+ |X_t/a|)^2} \right) dt 
\\ & = \frac{C}{\lambda} \mathbb{E} \left[ \left(1 +  ((\log \mathbf{e}_{\lambda})/2 - \log a + (\log \chi_2)/2)^+ \right)^{-2} \right]
\\ & = \frac{C}{\lambda} \mathbb{E} \left[ \left(1 +  (V- \log \lambda)^+/2 \right)^{-2} \right]
\\ & \leq \frac{C}{\lambda} \mathbb{P} ( V < (\log \lambda) / 2) 
+ \frac{C}{\lambda} \left(1 + (- \log \lambda)/4\right)^{-2}
= \mathcal{O} (\lambda^{-1} (-\log \lambda)^{-2}). 
\end{align*}
Adding the estimates above for $j \in \{1,2,3\}$, 
we observe that the term $j = 1$ dominates, which gives 
$$\int_0^{\infty} e^{-\lambda t} W^{(2)} (F_t) dt
\sim_{\lambda \rightarrow 0} 
\frac{2 \pi}{\lambda (- \log \lambda)} \mathbf{W}^{(2)}\left[ F_{\infty} \right]. 
$$
Now, Karamata's Tauberian theorem (see, for example, \cite{BGT87}, Theorem 1.7.1.) implies the following: 
if $U$ is a nonincreasing function from $[0, \infty)$ to $[0, \infty)$, and if 
\[
\int_0^\infty e^{-\lambda t} U(t) dt \sim_{\lambda > 0, \lambda \rightarrow 0} \lambda^{-\rho} L(1/\lambda), 
\]
for some fixed $\rho \geq 0$ and for a slowly varying positive function $L$, i.e. 
$L(ct)/L(t) \rightarrow 1$ for any fixed $c > 0$ and $t \rightarrow \infty$,  then

$$
U(t) \sim_{t \to \infty} \frac{1}{\Gamma(\rho)} t^{\rho-1} L(t).$$
Applying this result to 
$U(t) = W^{(2)} (F_t)$, $\rho = 1$ and 
$L(t) = (2 \pi /\log t) \mathbf{W}^{(2)}\left[ F_{\infty} \right]$ gives 
$$W^{(2)} (F_t) \sim_{t \rightarrow \infty} (2 \pi /\log t) \mathbf{W}^{(2)}\left[ F_{\infty} \right],$$
which completes the proof of the proposition. 
\end{proof}
From now, we use the following notations: for $s \geq 0$ and $\omega \in\mathcal{C}\big([0,s]\to\mathbb{C}\big)$, we define by $(F^{(\omega)}_t)_{t \geq 0}$ the process
given by 
$$F^{(\omega)}_t = F_{t + s} ( \omega')$$
where 
$\omega'_u = \omega_u + X_0$ for $u \leq s$ and 
$\omega'_u = \omega_s + X_{u-s}$ for $u \geq s$. 
We get the following lemma: 
\begin{lemma} \label{lemmaclassC}
If $(F_t)_{t \geq 0}$ belongs to the class $\mathcal{C}$, then $(F^{(\omega)}_t)_{t \geq 0}$ also belongs to $\mathcal{C}$, for almost all functions 
$\omega \in \mathcal{C}\big([0,s],\mathbb{C}\big) $, with respect to the Wiener measure.  
\end{lemma}
\begin{proof}
It is clear that $(F^{(\omega)}_t)_{t \geq 0}$ is nonnegative and nonincreasing. 
Moreover, for $t \geq 0$, $F^{(\omega)}_t$ is measurable with respect to the $\sigma$-algebra generated by $\omega_u + X_0$ for $u \leq s$ and $\omega_s + X_{u-s}$ for $s \leq u \leq t +s$, and then it is $\mathcal{F}_t$-measurable. Finally, for some $a > 0$, $(F^{(\omega)}_t)_{t \geq 0}$ remains constant from the supremum of $t$ such that $|F_{t + s} ( \omega')| \leq a$, i.e. $|\omega_s + X_{t}| \leq a$. 
Hence, $(F^{(\omega)}_t)_{t \geq 0}$ is constant after 
$\sigma_{a + |\omega_s|}$. It only remains to prove that 
$\mathbf{W}^{(2)}(F^{({\omega})}_{\infty})<\infty$
for almost all $\omega \in \mathcal{C}\big([0,s], \mathbb{C}\big) $. 
We apply Proposition \ref{MtF} to $t = s$ and $Y_s = 1$, and we get that 
$$W^{(2)}(M_s^{(2)} (F_{\infty})) 
= \mathbf{W}^{(2)} (F_{\infty}) < \infty$$
and then $W^{(2)}$-almost surely
$$\mathbf{W}^{(2)}(G) = M_s^{(2)} (F_{\infty}) < \infty.$$
where for all $\omega \in \Omega$ such that $\omega(0) = 0$, 
$$G((\omega_u)_{u \geq 0}) = F_{\infty} ((\omega'_u)_{u \geq 0})$$
where $\omega'_u = X_u$ for $0 \leq u \leq s$, and 
$\omega'_u = X_s + \omega_{u-s}$ for $u \geq s$. 
In other words,
for Wiener-almost every $\omega \in \mathcal{C}([0,s], \mathbb{C} )$, 
$$\mathbf{W}^{(2)}(G^{(\omega)}) < \infty$$
where 
$$G^{(\omega)}((\eta_u)_{u \geq 0}) = F_{\infty} ((\eta'_u)_{u \geq 0})$$
for $\eta_0 = 0$, $\eta'_u = \omega_u$, $0 \leq u \leq s$, and 
$\eta'_u = \omega_s + \eta_{u-s}$ for $u \geq s$. 
Now, we are done, since
$$G^{(\omega)} ((X_u)_{u \geq 0})
= F^{(\omega)}_{\infty}$$
as soon as $X_0 = 0$, which occurs almost everywhere under $\mathbf{W}^{(2)}$. 
\end{proof}
We can now quickly complete the proof of Theorem \ref{TheMainResult}. 
For $s \geq 0$, and for Wiener-almost every $\omega \in \mathcal{C}([0,s], \mathbb{C})$, we have, from Proposition \ref{limitlogt} 
$$\frac{\log (t-s)}{2 \pi} W^{(2)} (F^{(\omega)}_{t-s})
\underset{t \rightarrow \infty}{\longrightarrow} \mathbf{W}^{(2)} 
(F^{(\omega)}_{\infty})$$
since $(F^{(\omega)}_{u})_{u \geq 0}$ belongs to class 
$\mathcal{C}$, by Lemma \ref{lemmaclassC}. On the other hand
$$\frac{\log t}{2 \pi} W^{(2)} (F_{t})
\underset{t \rightarrow \infty}{\longrightarrow} \mathbf{W}^{(2)} 
(F_{\infty}),$$
the limit being strictly positive by assumption
on $(F_t)_{t \geq 0}$. 
Since $X_0 = 0$ almost surely under $W^{(2)}$, we have almost surely, by 
Markov property, 
$$W^{(2)} (F_t | \mathcal{F}_s) 
= W^{(2)} (F^{(\omega)}_{t-s})$$
where $\omega$ is the random function given by the canonical process up to time $s$. 
We deduce 
$$\frac{W^{(2)} (F_t | \mathcal{F}_s) }{W^{(2)} (F_t)} 
\underset{t \rightarrow \infty}{\longrightarrow}  
\frac{\mathbf{W}^{(2)} 
(F^{(\omega)}_{\infty})}{\mathbf{W}^{(2)} 
(F_{\infty})}. 
$$
Now, the random functional $G$ in the proof of Lemma \ref{lemmaclassC} coincides with the functional 
$F^{(\omega)}_{\infty}$ when $\omega$ is given 
by the canonical process up to time $s$, as soon as the functionals are applied to trajectories starting from $0$. 
We then get, almost surely, 
$$\frac{W^{(2)} (F_t | \mathcal{F}_s) }{W^{(2)} (F_t)} 
\underset{t \rightarrow \infty}{\longrightarrow}  
\frac{\mathbf{W}^{(2)} 
(G)}{\mathbf{W}^{(2)} 
(F_{\infty})} = \frac{M^{(2)}_s(F_{\infty})}{\mathbf{W}^{(2)} 
(F_{\infty})}. $$
By Fatou's lemma, for any event $\Gamma_s$ in $\mathcal{F}_s$, 
$$\underset{t \rightarrow \infty}{\lim \inf}  \, \mathbb{Q}_t (\Gamma_s) 
= \underset{t \rightarrow \infty}{\lim \inf} \, W^{(2)} 
\left( \frac{\mathbf{1}_{\Gamma_s} F_t }{W^{(2)} (F_t)}  \right) 
= \underset{t \rightarrow \infty}{\lim \inf} \, W^{(2)} 
\left( \mathbf{1}_{\Gamma_s} \frac{ W^{(2)} (F_t | \mathcal{F}_s)}{W^{(2)} (F_t)}  \right) 
\geq \frac{W^{(2)} (\mathbf{1}_{\Gamma_s}M_s^{(2)}(F_{\infty}))}{
\mathbf{W}^{(2)} 
(F_{\infty})}.
$$ 
Hence, 
\begin{equation}\underset{t \rightarrow \infty}{\lim \inf}  \, \mathbb{Q}_t (\Gamma_s)  \geq W_{\infty}^F (\Gamma_s), \label{Fatou}
\end{equation}
since 
\begin{equation}W_{\infty}^F (\Gamma_s)
= \frac{\mathbf{W}^{(2)} (F_{\infty} \mathbf{1}_{\Gamma_s} )} {\mathbf{W}^{(2)} (\Gamma_s)}
= \frac{W^{(2)} (M_s^{(2)} (F_{\infty})  \mathbf{1}_{\Gamma_s} )} {\mathbf{W}^{(2)} (F_{\infty})}
\label{densityWinftyF}
\end{equation}
by Proposition \ref{MtF} applied to $Y_s = \mathbf{1}_{\Gamma_s}$. 
Applying \eqref{Fatou} to the complement of $\Gamma_s$ 
and subtracting from $1$, we get 
\begin{equation}\underset{t \rightarrow \infty}{\lim \sup}  \, \mathbb{Q}_t (F_{\infty})  \leq W_{\infty}^F (\Gamma_s).
\end{equation}
Combining with \eqref{Fatou} gives the desired convergence. Formula \eqref{densityWinftyF} provides the density of the restriction of $W_{\infty}^F$ to $\mathcal{F}_s$ with respect to the Wiener measure. 
\section{Further direction of research}
\label{discussion}

The present work focuses on penalisation limits for two-dimensional Brownian motion. While the general penalisation framework developed by Najnudel, Roynette, and Yor \cite{NRY09} applies in considerable generality, the identification of explicit limiting measures in dimension two and higher remains largely open for diffusion processes and discrete-time Markov chains.
\\\\
The approach developed here exploits structural features specific to planar Brownian motion, most notably the explicit path decomposition at the last exit time from a bounded domain and the associated description of the $\sigma-$finite universal measure $\mathbf{W}^{(2)}$. These ingredients allow for a precise analysis of the long-time behaviour of a natural class of penalisation functionals.
\\\\
Within this framework, we restrict attention to nonnegative, adapted penalisation processes that are decreasing along trajectories and that stabilise after the Brownian path exits a bounded region. These assumptions ensure both the finiteness of the relevant normalising constants and the applicability of Laplace transform and Tauberian techniques. penalisation functionals exhibiting long-range spatial dependence or persistent time inhomogeneity fall outside the scope of the present method.
\\\\
It would be of interest to determine to what extent these structural assumptions can be relaxed, either by developing alternative decompositions of the underlying path or by identifying different universal measures governing the asymptotic behaviour. Such extensions may provide a route toward treating more general diffusion processes and discrete Markov models within a two-dimensional penalisation framework.

\bibliographystyle{alpha}    
\bibliography{penalization}
 
\end{document}